\documentclass[a4paper,12pt]{preprint}
\usepackage[margin=1.3in]{geometry}  
\usepackage[full]{textcomp}
\usepackage[osf]{newtxtext}

\usepackage{mathalfa}
\usepackage{microtype}

\usepackage{booktabs}
\usepackage{amsmath}
\usepackage{amsfonts}
\usepackage{amssymb}

\usepackage{enumitem}
\usepackage{amsmath}
\usepackage{amsfonts} 
\usepackage{amssymb}             

\usepackage{margincomment}

\usepackage{mathrsfs}
\usepackage{booktabs}
\usepackage{mathtools}

\usepackage{tikz}
\usepackage{amsthm}                
\usepackage{mhequ}
\usepackage{hyperref}

\hypersetup{pdfstartview=}

\addtocontents{toc}{\protect\hypersetup{hidelinks}}

\DeclareSymbolFont{sfoperators}{OT1}{ptm}{m}{n}
\DeclareSymbolFontAlphabet{\mathsf}{sfoperators}

\makeatletter
\def\operator@font{\mathgroup\symsfoperators}
\makeatother

\numberwithin{equation}{section}

\newtheorem{thm}{Theorem}[section]
\newtheorem{defn}[thm]{Definition}
\newtheorem{lem}[thm]{Lemma}
\newtheorem{prop}[thm]{Proposition}

\newtheorem{cor}[thm]{Corollary}

\theoremstyle{remark}
\newtheorem{remark}[thm]{Remark}
\newtheorem{rmk}[thm]{Remark}

\definecolor{darkgreen}{rgb}{0.1,0.7,0.1}
\definecolor{darkred}{rgb}{0.7,0.1,0.1}
\definecolor{darkblue}{rgb}{0,0,0.7}
\newcommand{\NN}{\mathbb{N}}

\newcommand{\RR}{\mathbb{R}}      

\newcommand{\ZZ}{\mathbb{Z}}      

\newcommand{\cC}{\mathcal{C}}

\newcommand{\fF}{\mathcal{F}}

\newcommand{\nN}{\mathcal{N}}
\newcommand{\oO}{\mathcal{O}}
\newcommand{\pP}{\mathcal{P}}

\newcommand{\tT}{\mathcal{T}}

\makeatletter 
\newcommand{\cov}{{\operator@font cov}}
\newcommand{\var}{{\operator@font var}}
\newcommand{\corr}{{\operator@font corr}}
\newcommand{\diam}{{\operator@font diam}}
\newcommand{\Av}{{\operator@font Av}}
\newcommand{\trig}{{\operator@font trig}}
\newcommand{\Enh}{{\operator@font Enh}}
\newcommand{\EEnh}{\overline {\operator@font Enh}}
\newcommand{\Com}{{\operator@font Com}}
\makeatother

\newcommand{\C}{\mathbf{C}}
\newcommand{\E}{\mathbf{E}}

\newcommand{\Z}{\mathbf{Z}}
\newcommand{\0}{\mathbf{0}}
\newcommand{\1}{{\color{black}{\mathbf{1}}}}

\colorlet{symbols}{blue!90!black}
\colorlet{testcolor}{green!60!black}

\def\${|\!|\!|}

\definecolor{darkgreen}{rgb}{0.1,0.7,0.1}
\definecolor{darkred}{rgb}{0.7,0.1,0.1}
\definecolor{darkblue}{rgb}{0,0,0.7}
\makeatletter
\DeclareRobustCommand{\TitleEquation}[2]{\texorpdfstring{\StrLeft{\f@series}{1}[\@firstchar]$\if%
b\@firstchar\boldsymbol{#1}\else#1\fi$}{#2}}
\makeatother

\begin{document}

\title{A functional Breuer--Major theorem with Poisson noise and a sharp regularity condition}

\author{Fanhao Kong$^1$ and Haiyi Wang$^2$}
\institute{NYU Shanghai, China. \email{fk2601@nyu.edu}
\and
Peking University, China. \email{wanghaiyi@stu.pku.edu.cn}
}

\maketitle

\begin{abstract}
    We extend the functional Breuer--Major theorem to stationary sequences generated from a Poisson point process under a regularity condition on the nonlinearity. We use the $L^p$ spectral gap inequality of the Poisson point process to prove the tightness. Our use of the spectral gap inequality naturally requires a regularity condition. In contrast to the Gaussian case, this condition cannot in general be replaced by integrability alone.  We construct an example at the critical regularity showing that the regularity condition is necessary and the regularity threshold is sharp.
\end{abstract}

\setcounter{tocdepth}{2}
\tableofcontents

\section{Introduction}

\subsection{The Breuer-Major theorem for Gaussians}

Let $X$ be a Gaussian random variable and $\phi: \RR \rightarrow \RR$ such that $\phi(X)$ has finite second moment. Then the random variable $\phi(X)$ can be decomposed into an $L^2$ orthogonal sum
\begin{equation*}
    \phi(X) = \sum_{k \geq 0} \pP_k \big( \phi(X) \big)\;,
\end{equation*}
where $\pP_k$ denotes the projection onto the $k$-th homogeneous Wiener chaos. For each integer $d \geq 0$, define
\begin{equation*}
    \tT^{\geq d} \big( \phi(X) \big) := \sum_{k \geq d} \pP_k \big( \phi(X) \big)\;.
\end{equation*}
Now, let $X_1, X_2, \dots$ be a stationary sequence of $\nN(0,1)$ random variables and
\begin{equation} \label{e:process}
    Y_n(t) \coloneqq \frac{1}{\sqrt{n}} \sum_{j=0}^{\lfloor nt \rfloor - 1} \phi(X_j)
\end{equation}
for some $\phi: \RR \rightarrow \RR$ with $\phi(X_j)$ having finite second moment. The classical Breuer-Major theorem (\cite{breuer_clt}) states that if the correlations of the stationary Gaussian sequence satisfy
\begin{equation} \label{eq:hypothesis1}
    \sum_{j \in \ZZ} \left| \E (X_0 X_j) \right|^d  < +\infty
\end{equation}
for some integer $d \geq 1$, then as $n \rightarrow +\infty$, the finite-dimensional distributions of the process $\tT^{\geq d} \big( Y_n(t) \big)$ converge to those of a Brownian motion. 

To prove tightness, one can use $L^p(\Omega)$ estimates for $\tT^{\ge d}Y_n(t_2)-\tT^{\ge d}Y_n(t_1)$ for some $p>2$. By assuming very strong decay conditions on the coefficients of the chaos expansion of $\phi$, \cite{chambers1989central} and \cite{hariz2002limit} proved tightness of the process in the Skorohod space $\mathbf{D}([0,1])$. By using hypercontractivity to bound the \(L^p\)-norms of each chaos component of $\tT^{\geq d} Y_n$, one can see that this is the condition that guarantees summability of the series. This condition is too restrictive for the statement to be interesting. 

The reason for such a procedure far from being optimal is that by summing up the \(L^p\)-norms of each term, one does not take into account the many cancellations between different chaos components. Hence, in $L^p$ setting for $p>2$, one needs to consider $\tT^{\geq d} \big( Y_n(t) \big)$ as a whole object instead of treating individual chaos components separately. \cite{nualart_breuer_clt} achieved this by employing Meyer inequality (see \cite[Theorem 1.5.1]{Nua06}) in Malliavin calculus. \cite{campese2020continuous} achieved this in the same way in the continuous version of functional Breuer–Major theorem. More precisely, they proved tightness of the process $\tT^{\geq d} \big( Y_n(t) \big)$ in the Skorohod space $\mathbf{D}([0,1])$ merely by assuming that $\phi(X_0) \in L^p(\Omega)$ for some $p>2$ .

Recently, there are several extensions of the functional Breuer--Major theorem in the Gaussian case. \cite{MansanarezPolyZheng2026,MilesisZheng2026} reduced the integrability requirement for tightness of $\tT^{\ge d}Y_n$ in $\mathbf D([0,1])$ to $\phi(X_0)\in L^2(\Omega)$. Assuming that $\phi(X_0)$ and its Malliavin derivatives up to order $d$ belong to $L^4(\Omega)$, \cite{EladAltmanKlosePerkowski2026} proved joint convergence of the first and second order processes in rough path topology. \cite{BenningNourdin2026} obtained an analogous result under $\phi(X_0)\in L^p(\Omega)$ for some $p>6$.
\par

A natural question is to generalize this central limit theorem to the non-Gaussian case. In this paper, we establish a functional central limit theorem for the process \eqref{e:process} when $X_n$ arises from a Poisson point process. Our proof is based on the following two steps.

The first step is the tightness of $\{\tT^{\geq d}  Y_n\}_{n\geq1}$. 
The main challenge to apply the Malliavin calculus approach in \cite{nualart_breuer_clt} and \cite{campese2020continuous} here is the lack of Meyer inequality (see \cite[Theorem 1.5.1]{Nua06}) in the Malliavin calculus of the Poisson case. To overcome this issue, we use the $L^p$ spectral gap inequality to control the moments of the stochastic process. This technique has parallels in recent studies of stochastic partial differential equations (see \cite{Otto_spectral_gap}, \cite{BPHZ_spectral_gap} and~\cite{KWX24}).

The second step is the convergence of finite-dimensional distributions. 
The direct application of the classical method of moments from \cite{breuer_clt} fails in the Poisson case due to the requirement of bounds for every derivative of $\phi$. We truncate the kernel $\psi$ and apply a central limit theorem for sequences with finite dependence range.

Our theorem imposes a regularity condition on $\phi$. In view of the Gaussian results in \cite{nualart_breuer_clt,MilesisZheng2026}, one may expect that this condition can be replaced by an integrability condition on $\phi(X_0)$ alone. However, we construct a kernel $\psi$ satisfying the decay condition and a bounded function $\phi$ at the critical regularity such that $(\tT^{\ge d}Y_n(1))_{n\ge1}$ is not tight. Thus this regularity condition is sharp and cannot in general be replaced by integrability alone.

\subsection{Main result}
    Let $(\Omega, \mathcal{F},\mathbf P)$ be a probability space and let $\eta$ denote a Poisson point process on $\RR$ with Lebesgue intensity measure. For $g\in L^{2}(\mathbb{R}^n)$, we define $I_n(g)$ as the multiple Wiener-It\^o integral of order $n$ with respect to $\eta$.
    
    For $m \in \mathbb N$ and $F \in L^2(\Omega)$, the truncation operator $\tT^{\geq m}$ is defined by
    \begin{equation} \label{e:op_geq}
        \mathcal{T}^{\ge m} F = \sum_{n=m}^\infty I_n(f_n),
    \end{equation}
    where $F=\sum_{n=0}^\infty I_n(f_n)$ represents the chaos expansion of $F$. 

    Given real valued $\psi\in L^2(\RR)$, we construct the stationary sequence
    \begin{equation} \label{e:poisson}
        X_j \coloneqq I_1(\psi_j),
    \end{equation}
    where $\psi_j(x) = \psi(x-j)$. Define
    \begin{equation*}
    	Y_n(t) \coloneqq \frac{1}{\sqrt{n}} \sum_{j=0}^{[nt] - 1} \phi(X_j), \quad t \in [0,1].
	\end{equation*}
Before stating our main result, we first introduce a class of functions with specific regularity properties.
\begin{defn} \label{def:nonlinearity_space}
    Fix $M>0$. For $\gamma>0$ with $\gamma\notin\NN$, the class $\cC_{M}^{\gamma}$ consists of functions $\phi: \RR \rightarrow \RR$ such that
    \begin{equation*}
        \|\phi\|_{\cC_M^\gamma}\coloneqq\sup_{\substack{x\in\RR\\0\le\ell\le[\gamma]}}\!\frac{|\phi^{(\ell)}(x)|}{(1+x^2)^{\frac{M}{2}}}+\sup_{\substack{x\in\RR\\0<|h|<1}}\!\frac{|\phi^{([\gamma])}(x+h)-\phi^{([\gamma])}(x)|}{|h|^{\gamma-[\gamma]}(1+x^2)^{\frac{M}{2}}}<\infty.
    \end{equation*}
    For $\gamma\in\NN^+$, the class $\cC_M^\gamma$ consists of functions $\phi: \RR \rightarrow \RR$ such that
    \begin{equation*}
        \|\phi\|_{\cC_M^\gamma}\coloneqq\sup_{\substack{x\in\RR\\0\le\ell\le\gamma-1}}\!\frac{|\phi^{(\ell)}(x)|}{(1+x^2)^{\frac{M}{2}}}+\sup_{\substack{x\in\RR\\0<|h|<1}}\!\frac{|\phi^{(\gamma-1)}(x+h)-2\phi^{(\gamma-1)}(x)+\phi^{(\gamma-1)}(x-h)|}{|h|(1+x^2)^{\frac{M}{2}}}<\infty.
    \end{equation*}
\end{defn}

Now we are ready to state the functional central limit theorem for $\mathcal{T}^{\ge d} Y_n$. We write $q_\alpha=\frac{1}{2\alpha-1}$.
\begin{thm} \label{thm}
    Fix $d\in\NN^+$, $\alpha > \frac{1}{2} + \frac{1}{2d}$ and $\gamma>q_\alpha$. Suppose that there exist $C_\psi>0$ and $M\in \NN$ such that
    \begin{equation}\label{as:psi}  
        |\psi(x)| \le C_\psi (1 + |x|)^{-\alpha}
    \end{equation}
        for all $x\in\RR$ and
    \begin{equation}\label{as:phi}  
        \phi \in \cC^{\gamma}_{M}.
    \end{equation}
    Then the process $\mathcal{T}^{\ge d} Y_n$ converges in law to $\mu B$ in the Skorohod space $\mathbf{D}([0,1])$, where $(B_t)_{t \in [0,1]}$ is a standard Brownian motion and $\mu\ge0$ is given by
	\begin{equation} \label{e:sigma}
    	\mu^2 = \lim_{n\to\infty}\frac1n\mathbf E\biggl|\sum_{j=0}^{n-1}\mathcal T^{\ge d}\phi(X_j)\biggr|^2.
	\end{equation}
\end{thm}

\begin{proof}
	The result follows from the tightness in the Skorohod space ${\bf D}([0,1])$ in Theorem~\ref{thm:tightness} and the convergence of finite-dimensional distributions in Theorem~\ref{thm:fdd_convergence}.
\end{proof}

    To ensure the tightness of $(\mathcal{T}^{\ge d} Y_n)_{n\ge1}$, we need to control the $p$-th moment of the process for some $p>2$ (see Lemma~\ref{lem:tightness}). The proof of the tightness in \cite{nualart_breuer_clt} relies on the Malliavin calculus. In the Poisson case, we use the $L^p$ spectral gap inequality to control the $p$-th moment of the process in terms of the $p$-th moment of the mixed $L^2$ and $L^p$ norms of higher order Malliavin derivatives. Since the Malliavin derivative in the Poisson case is the discrete derivative,  we additionally require the regularity assumption \eqref{as:phi} to control the Malliavin derivative of $\phi(X_j)$ (see \eqref{eq:LP-differences}). To obtain the optimal condition $\gamma>q_\alpha$, we decompose $\phi$ into Littlewood--Paley blocks and interpolate the $L^2(\Omega)$ and $L^P(\Omega)$ estimates for each block with $P>2$.

A parallel result holds for the linearly interpolated process $Z_n$ in $\C ([0,1])$, where $Z_n$ is defined by
\begin{equation*}
    Z_n(t)\coloneqq Y_n(t) + \frac{nt-[nt]}{\sqrt{n}} \phi(X_{[nt]}),\qquad t\in[0,1].
\end{equation*}
\begin{thm} \label{thm:continuous}
    Under the same assumptions as in Theorem~\ref{thm}, the process $\mathcal{T}^{\ge d} Z_n$ converges in law to $\mu B$ in $\C ([0,1])$, where $(B_t)_{t \in [0,1]}$ is a standard Brownian motion and $\mu$ is given by \eqref{e:sigma}.
\end{thm}

\begin{proof}
    The result follows from the tightness in Theorem~\ref{thm:tightness_z} and the convergence of the finite-dimensional distributions in Corollary~\ref{cor:z_fdd}.
\end{proof}

The following proposition ensures that our regularity assumption \eqref{as:phi} is necessary and sharp.
\begin{prop}\label{prop:sharpness}
    Fix $d\in\NN^+$ and $\alpha>\frac12+\frac{1}{2d}$. Then there exist $\psi\in L^2(\RR)$ satisfying \eqref{as:psi} and a bounded function $\phi\in\cC_0^{q_\alpha}$ such that the family $(\tT^{\ge d}Y_n(1))_{n\ge1}$ is not tight in $\RR$.
\end{prop}

\subsection*{Notations}

The relation $A\lesssim B$ means that there exists a constant $C$ such that $A\leq CB$ and $C$ is independent of some parameters.

\subsection*{Organisation of this article}
    In Section~\ref{sec:preliminaries}, we introduce some preliminary lemmas in Malliavin calculus for Poisson point process and the Littlewood--Paley decomposition. In Section~\ref{sec:tightness}, we establish the tightness of $(\mathcal{T}^{\ge d} Y_n)_{n\ge1}$ and $(\mathcal{T}^{\ge d} Z_n)_{n\ge1}$, which is based on the $L^p$ spectral gap inequality for the Poisson point process. In Section~\ref{sec:fdd_convergence}, we show the convergence of finite-dimensional distributions by kernel truncation and using the central limit theorem for the finite-range dependent sequence. In Section~\ref{sec:sharpness}, we construct an example to show Proposition~\ref{prop:sharpness}.
    
\section{Preliminaries}\label{sec:preliminaries}

\subsection{Spectral gap inequality for Poisson point process}
    
     We introduce some basic elements of Malliavin calculus for Poisson point processes, most of which can be found in \cite{LecturesonPP}. Let $(\Omega, \mathcal{F},\mathbf P)$ be a probability space. Let $\mathbf{N}_\sigma$ denote the set of $\NN\cup\{\infty\}$-valued $\sigma$-finite measures on $\RR$ and let $\mathcal{N}_\sigma$ denote the smallest $\sigma$-algebra that makes the mapping $\mathbf{N}_\sigma \ni \xi \mapsto \xi(W)$ measurable for every measurable set $W\subset\RR$. Let $\eta$ denote the Poisson point process on $\RR$ with Lebesgue intensity measure. We define
\begin{equation*}
    L^0_\eta = \{ f(\eta) \;|\; f : \mathbf{N}_\sigma \rightarrow \mathbb{C}\text{ is a } \mathcal{N}_\sigma\text{-measurable function}\; \}.
\end{equation*}
    For every $p>0$, denote $L^p_\eta$ as the set of random variables in $L^0_\eta$ with finite $p$-th moment.

    For $F=f(\eta) \in L^0_\eta$ and $x \in \RR$, the difference operator $D_x$ is defined by
        \begin{equation}\label{eq:difference-definition}
            D_x  F \coloneqq  f(\eta + \delta_x) -  f(\eta),
        \end{equation}
    where $\delta_x$ is the Dirac mass at $x$. For $n \in \NN$, the $n$-th order difference $D^n$ at ${\vec{x}} \in \RR^n$ is defined by
    \begin{equation*}
        D^n_{\vec{x}} F \coloneqq D_{x_n} D_{x_{n-1}} \dots D_{x_1} F\;,
    \end{equation*}
    where $\vec{x} = (x_1, \dots, x_n) \in \RR^n$.

    For the measure $\chi = \sum_{j} \delta_{x_j}$, its $k$-th factorial measure $\chi^{\diamond k}$ is a sum of delta measures on $\RR^k$, defined by
    \begin{equation*}
        \chi^{\diamond k} \coloneqq \sum \delta_{(x_{j_1}, \dots, x_{j_k})}\;,
    \end{equation*}
    where the sum is taken over $j_1, \dots, j_k$ such that $j_i \neq j_{i'}$ for $i\neq i'$. In short, $\chi^{\diamond k}$ is the $k$-th direct product of $\chi$ excluding repeated points. 
    
    For $n \in \NN^+$ and $p \geq 1$, let $L_s^p(\RR^n)$ denote the space of complex-valued $L^p(\RR^n)$ functions that are symmetric under permutations of their $n$ variables. The $n$-th Wiener-It\^o multiple integral
    \begin{equation*}
        I_{n}: L_s^1(\RR^n) \cap L_s^2(\RR^n)\rightarrow L^2(\Omega)
    \end{equation*}
    is defined by
    \begin{equation*}
        I_n(g) = \sum_{k=0}^{n} (-1)^{n-k} \binom{n}{k} \int_{\RR^{n-k}} \int_{\RR^k} g \, d\eta^{\diamond k} \, dx.
    \end{equation*}
    According to~\cite[Proposition~12.9 ]{LecturesonPP}, $ I_n$ extends uniquely to a map from $L_s^2(\RR^n)$ to $L^2(\Omega)$ with $\E I_{n} (f) = 0$ for $n \geq 1$. 
    
    According to~\cite[Theorem~18.10]{LecturesonPP}, the chaos expansion of $F \in L^2_\eta$ is given by
    \begin{equation}\label{eq:WienerItoChaos}
        F=\sum_{n=0}^\infty I_n(f_n),
    \end{equation}
    where $f_n(\vec{x}) = \frac{1}{n!} \mathbf{E}D^n_{\vec{x}}F\in L_s^{2}(\mathbb{R}^n)$ and the series converges in $L^2(\Omega)$. This expansion satisfies the isometry
    \begin{equation} \label{e:chaos_expansion} 
	    \| F\|^2_{L^2(\Omega)} = \sum_{n=0}^\infty n! \left\|f_n\right\|^2_{L^2(\RR^n)}\;.
    \end{equation}
    For $G\in L^2_\eta$ with the chaos expansion $G=\sum_{n=0}^\infty I_n(g_n)$,
    we have (see \cite[Theorem 18.6]{LecturesonPP})
    \begin{equation}\label{e:iso}
        \E F\overline{G}=\sum_{n=0}^\infty n! \langle f_n,g_n\rangle_{L^2(\RR^n)}\;,
    \end{equation}
    where $L^2(\RR^n)$ is a complex Hilbert space.
    According to \cite[Theorem~3.3]{chaosexpansion}, for $F\in L^2_\eta$ satisfying
    \begin{equation*}
        \sum_{n=1}^{\infty} n \cdot n!\left\|f_n\right\|^2_{L^2(\RR^n)} < +\infty\;,
    \end{equation*}
    the chaos expansion of $D_xF$ is
    \begin{equation}\label{DxFchaos}
        D_x F = \sum_{n=1}^\infty n I_{n-1} \left( f_n(x,\cdot) \right)\;.
    \end{equation}
    This property implies that the difference operator $D_x$ is the same as the Malliavin derivative for the Poisson point process.

\begin{lem}\label{lem:finite-difference-chain-rule}
Let $k\in\NN^+$, $G\in C^k(\RR,\RR)$ and $j\in\ZZ$. Then we have
\begin{equation}\label{eq:finite-difference-chain-rule}
 D^k_{\vec x}G(X_j)=\int_{[0,1]^k}G^{(k)}\biggl(X_j+\sum_{i=1}^kt_i\psi_j(x_i)\biggr)\biggl(\prod_{i=1}^k\psi_j(x_i)\biggr)\,d\vec t,
\end{equation}
where $\vec x=(x_1,\ldots,x_k)\in\RR^k$.
In particular, for $G(x)=x$ and $k=1$, we have
\begin{equation*}
 D_xX_j=\psi_j(x).
\end{equation*}
\end{lem}
\begin{proof}
By the definition of $D_x$ and \eqref{DxFchaos}, for $k=1$ we have
\begin{equation*}
 D_xG(X_j)=G(X_j+\psi_j(x))-G(X_j).
\end{equation*}
The case $k=1$ follows from the fundamental theorem of calculus, and the cases $k\ge2$ follow by induction.
\end{proof}

We will use the following $L^p$ spectral gap inequality for the Poisson point process.

\begin{prop}\cite[Proposition~3.8]{KWX24}\label{prop:sg_ineq}
    For every $p \ge 2$ and $F \in L_\eta^1$, we have
    \begin{equation*}
        \|F\|_{L^p(\Omega)} \lesssim |\mathbf E F| + \|\|D_x F\|_{L^p(\Omega)}\|_{L^2(\RR)} + \|\|D_x F\|_{L^p(\Omega)}\|_{L^p(\RR)},
    \end{equation*}
    where the proportionality constant is independent of $F$.
\end{prop}
\begin{cor}\label{cor:multi_sg_ineq}
    For every $p \ge 2$, $d\ge 1$ and $F \in L_\eta^2$, we have
    \begin{equation*}
        \|\tT^{\ge d}F\|_{L^p(\Omega)} \lesssim \sum_{v=0}^{d} \Big\|\Big\| \|D^{d}_{\vec{x}} F\|_{L^{p}(\Omega)}\Big\|_{L^{p}(d\vec{y}_v)}\Big\|_{L^2(d\vec{z}_v)},
    \end{equation*}
    where $\vec{x}=(x_1,\dots,x_{d})$, $\vec{y}_v = (x_1,\dots,x_v)$, $\vec{z}_v=(x_{v+1},\dots,x_{d})$ and the proportionality constant is independent of $F$.
\end{cor}
\begin{proof}
    According to \eqref{eq:WienerItoChaos}, we have $\E D^k_{\vec x}\tT^{\ge d}F =0$ for every $0\le k<d$. By Proposition~\ref{prop:sg_ineq}, we get
    \begin{equation}
        \|\tT^{\ge d}F\|_{L^p(\Omega)} \lesssim \|\|D_{x}(\tT^{\ge d} F)\|_{L^p(\Omega)}\|_{L^{p}(\RR)\cap L^2(\RR)}.
    \end{equation}
    We repeat this step $d-1$ times to the inner $L^{p}(\Omega)$-term and then obtain
    \begin{equation*}
        \|\tT^{\ge d}F\|_{L^p(\Omega)} \lesssim \|\cdots\|\|D^d_{\vec{x}}(\tT^{\ge d} F)\|_{L^p(\Omega)}\|_{L^{p}(\RR)\cap L^2(\RR)}\cdots\|_{L^{p}(\RR)\cap L^2(\RR)}.
    \end{equation*}
    Then the desired bound follows from $D^d_{\vec{x}}(\tT^{\ge d} F) = D^d_{\vec{x}} F$ and Minkowski's inequality.
\end{proof}

%

A direct application of Proposition~\ref{prop:sg_ineq} is the existence of $p$-th moment of the sequence $X_j$ given by \eqref{e:poisson}. 
\begin{lem}\label{lem:input-moments}
Under the assumption \eqref{as:psi}, for every $p\ge2$, we have
\begin{equation}\label{eq:LP-input-moments}
    \|X_j\|_{L^p(\Omega)}\lesssim1.
\end{equation}
\end{lem}
\begin{proof}
It follows directly from Proposition~\ref{prop:sg_ineq} and \eqref{as:psi} since $\E X_j=0$ and $D_xX_j=\psi_j(x)$.
\end{proof}

\subsection{Littlewood--Paley decomposition}\label{sec:LP}

We use $\fF$ to represent the Fourier transform, defined by $\mathcal Ff(\xi)=\int_\RR e^{-2\pi{\rm i}x\xi}f(x)\,dx$ on $L^1(\RR)$ and extended to tempered distributions. Following \cite[Section~2.2, Proposition~2.10]{BookChemin}, choose even functions $\chi,\varphi\in\cC_c^\infty(\RR)$ with values in $[0,1]$ such that
\begin{equation*}
    \operatorname{supp}\chi\subset\bigl[-\tfrac43,\tfrac43\bigr],\qquad
    \operatorname{supp}\varphi\subset\bigl\{\xi\in\RR:\tfrac34\le|\xi|\le\tfrac83\bigr\},
\end{equation*}
and
\begin{equation*}
    \chi(\xi)+\sum_{\ell=0}^\infty\varphi(2^{-\ell}\xi)=1,\qquad\xi\in\RR.
\end{equation*}
For a function $\phi$ of polynomial growth, we define its nonhomogeneous dyadic blocks by
\begin{equation}\label{eq:LP-definition}
    \Delta_{-1}\phi=\mathcal F^{-1}(\chi\fF\phi),\qquad
    \Delta_\ell\phi=\mathcal F^{-1}\bigl(\varphi(2^{-\ell}\cdot)\fF\phi\bigr),\qquad\ell\ge0.
\end{equation}

\begin{lem}\label{lem:weighted-LP-bounds}
For every $\gamma>0$, $M\ge0$, $k\in\NN$ and real tempered distribution $\phi$ on $\RR$, we have
\begin{equation}\label{eq:weighted-LP-norm}
    \|\phi\|_{\cC_M^\gamma}\lesssim\sup_{\ell\ge-1}2^{\gamma\ell}\sup_{x\in\RR}\frac{|(\Delta_\ell\phi)(x)|}{(1+x^2)^{\frac{M}{2}}}\lesssim\|\phi\|_{\cC_M^\gamma}
\end{equation}
and
\begin{equation}\label{eq:LP-derivatives}
    \sup_{x\in\RR}\frac{|(\Delta_\ell\phi)^{(k)}(x)|}{(1+x^2)^{\frac{M}{2}}} \lesssim2^{(k-\gamma)\ell}\|\phi\|_{\cC_M^\gamma}
\end{equation}
uniformly in $\ell\ge-1$.
\end{lem}
\begin{proof}
For $M=0$, \eqref{eq:weighted-LP-norm} follows from \cite[Section~3.5, Theorem~1]{Salo2008} if $\gamma\notin\NN$ and from \cite[Section~3.5, Theorem~2]{Salo2008} if $\gamma\in\NN^+$. The estimate \eqref{eq:LP-derivatives} follows from \cite[Lemma~2.1]{BookChemin}. The case $M>0$ follows similarly.
\end{proof}

\section{Tightness} \label{sec:tightness}

   In this section, we establish the tightness of $(\mathcal{T}^{\ge d} Y_n)_{n\ge1}$ in $\mathbf{D}([0,1])$. The corresponding result for $(\mathcal{T}^{\ge d} Z_n)_{n\ge1}$ will be presented in Section~\ref{sec:tightness_z}.

\subsection{Tightness of \TitleEquation{(\mathcal{T}^{\ge d} Y_n)_{n\geq 1}}{(mathcal{T} {ge d} Y n) {n geq 1}}}

\begin{thm}\label{thm:tightness}
    Under the assumptions of Theorem~\ref{thm}, the family $(\mathcal{T}^{\ge d} Y_n)_{n\ge1}$ is tight in the Skorohod space $\mathbf{D}([0,1])$.
\end{thm}

To prove Theorem~\ref{thm:tightness}, we rely on a sufficient condition for the tightness in $\mathbf{D}([0,1])$, which is stated in the following lemma.
\begin{lem}\cite[Lemma~3.1]{nualart_breuer_clt}\label{lem:tightness}
    If there exists $p>2$ such that
    \begin{equation} \label{e:tightness}
        \|\mathcal{T}^{\ge d} Y_n(t_2)- \mathcal{T}^{\ge d} Y_n(t_1)\|_{L^p(\Omega)} \lesssim \sqrt{\frac{[nt_2]-[nt_1]}{n}}
    \end{equation}
    holds uniformly in $0 \leq t_1 < t_2 \leq 1$ and $n$, then the family $(\mathcal{T}^{\ge d} Y_n)_{n\ge1}$ is tight in $\mathbf{D}([0,1])$.
\end{lem}

To prove the $L^p(\Omega)$ estimate \eqref{e:tightness} under the regularity condition $\gamma>q_\alpha$, we interpolate the $L^2(\Omega)$ and $L^P(\Omega)$ estimates for each Littlewood--Paley block with $P>2$. The following two lemmas provide these estimates.

For simplicity, we define
\begin{equation} \label{e:def:psi}
     \Psi_{\beta}(x)=(1+|x|)^{-\beta},\quad\Psi_{\beta,j}(x)=\Psi_{\beta}(x-j)
\end{equation}
for $x,\beta\in\RR$ and $j\in\ZZ$. These functions are designed to bound the moments of the Malliavin derivative of $(\Delta_\ell\phi)(X_j)$.

\begin{lem}\label{lem:sharp-frequency-block}
Suppose that the assumptions of Theorem~\ref{thm} hold. Then we have
\begin{equation}\label{eq:sharp-frequency-block}
    \biggl\|\sum_{j=[nt_1]}^{[nt_2]-1}\tT^{\ge d}(\Delta_\ell\phi)(X_j)\biggr\|_{L^2(\Omega)}\lesssim\sqrt{[nt_2]-[nt_1]}\,2^{-(\gamma-q_\alpha)\ell},
\end{equation}
uniformly in $\ell\ge-1$, $0\le t_1<t_2\le1$ and $n\in\NN^+$.
\end{lem}
\begin{proof}
By \eqref{eq:WienerItoChaos} and \eqref{e:chaos_expansion}, we have
\begin{equation*}
    \biggl\| \sum_{j=[nt_1]}^{[nt_2]-1} \tT^{\ge d} (\Delta_\ell\phi)(X_j) \biggr\|_{L^2(\Omega)}^2 = \sum_{k=d}^\infty \frac1{k!} \int_{\RR^k} \biggl| \sum_{j=[nt_1]}^{[nt_2]-1} \E D^k_{\vec x} (\Delta_\ell\phi)(X_j) \biggr|^2\,d\vec x,
\end{equation*}
where $\vec x=(x_1,\ldots,x_k)\in\RR^k$. For $R\ge1$, $k\ge d$ and $j\in\ZZ$, set $A_{j,R}^k=\{\vec x\in\RR^k:\min_{1\le i\le d}|x_i-j|\le R\}$. Splitting each kernel according to $A_{j,R}^k$ and its complement, we obtain
\begin{equation}\label{eq:kernel-block-L2}
\begin{split}
    \biggl\| \sum_{j=[nt_1]}^{[nt_2]-1} \tT^{\ge d} (\Delta_\ell\phi)(X_j) \biggr\|_{L^2(\Omega)}^2 &\lesssim \sum_{k=d}^\infty \frac1{k!} \int_{\RR^k} \biggl| \sum_{j=[nt_1]}^{[nt_2]-1} \E D^k_{\vec x} (\Delta_\ell\phi)(X_j) \1_{A_{j,R}^k}(\vec x) \biggr|^2\,d\vec x\\
    &\quad{}+ \sum_{k=d}^\infty \frac1{k!} \int_{\RR^k} \biggl| \sum_{j=[nt_1]}^{[nt_2]-1} \E D^k_{\vec x} (\Delta_\ell\phi)(X_j) \1_{(A_{j,R}^k)^c}(\vec x) \biggr|^2\,d\vec x.
\end{split}
\end{equation}
For the first term in \eqref{eq:kernel-block-L2}, by the Cauchy--Schwarz inequality, we have
\begin{align*}
    & \sum_{k=d}^\infty \frac1{k!} \int_{\RR^k} \biggl| \sum_{j=[nt_1]}^{[nt_2]-1} \E D^k_{\vec x} (\Delta_\ell\phi)(X_j) \1_{A_{j,R}^k}(\vec x) \biggr|^2\,d\vec x\\
    \le{} & \sum_{k=d}^\infty \frac1{k!} \int_{\RR^k} \biggl( \sum_{i=1}^d \sum_{j=[nt_1]}^{[nt_2]-1} \1_{\{|x_i-j| \le R\}} \biggr) \biggl( \sum_{j=[nt_1]}^{[nt_2]-1}| \E D^k_{\vec x} (\Delta_\ell\phi)(X_j)|^2 \biggr)\,d\vec x\\
    \le{} & d(2R+1) \sum_{j=[nt_1]}^{[nt_2]-1} \sum_{k=d}^\infty \frac1{k!} \int_{\RR^k}| \E D^k_{\vec x} (\Delta_\ell\phi)(X_j)|^2\,d\vec x.
\end{align*}
By \eqref{e:chaos_expansion}, \eqref{eq:LP-derivatives} and \eqref{eq:LP-input-moments}, we get
\begin{equation*}
\begin{split}
    &\sum_{j=[nt_1]}^{[nt_2]-1} \sum_{k=d}^\infty \frac1{k!} \int_{\RR^k}| \E D^k_{\vec x} (\Delta_\ell\phi)(X_j)|^2\,d\vec x\\
    = &\sum_{j=[nt_1]}^{[nt_2]-1}\| \tT^{\ge d} (\Delta_\ell\phi)(X_j)\|_{L^2(\Omega)}^2\\
    \le &\sum_{j=[nt_1]}^{[nt_2]-1} \E| (\Delta_\ell\phi)(X_j)|^2 \lesssim ([nt_2]-[nt_1]) 2^{-2\gamma\ell}.
\end{split}
\end{equation*}
Hence, we have
\begin{equation}\label{eq:kernel-near-L2}
    \sum_{k=d}^\infty \frac1{k!} \int_{\RR^k} \biggl| \sum_{j=[nt_1]}^{[nt_2]-1} \E D^k_{\vec x} (\Delta_\ell\phi)(X_j) \1_{A_{j,R}^k}(\vec x) \biggr|^2\,d\vec x \lesssim ([nt_2]-[nt_1]) R 2^{-2\gamma\ell}.
\end{equation}

For the second term in \eqref{eq:kernel-block-L2}, write $\vec y=(x_1,\ldots,x_d)$ and $\vec z=(x_{d+1},\ldots,x_k)$. Then we have
\begin{align}
    & \sum_{k=d}^\infty \frac1{k!} \int_{\RR^k} \biggl| \sum_{j=[nt_1]}^{[nt_2]-1} \E D^k_{\vec x} (\Delta_\ell\phi)(X_j) \1_{(A_{j,R}^k)^c}(\vec x) \biggr|^2\,d\vec x\notag\displaybreak[1]\\
    \le{} & \frac1{d!} \int_{\RR^d} \sum_{k=d}^\infty \frac1{(k-d)!} \int_{\RR^{k-d}} \biggl| \sum_{j=[nt_1]}^{[nt_2]-1} \E D^{k-d}_{\vec z}D^d_{\vec y} (\Delta_\ell\phi)(X_j) \1_{(A_{j,R}^d)^c}(\vec y) \biggr|^2\,d\vec z\,d\vec y\notag\\
    ={} & \frac1{d!} \int_{\RR^d} \biggl\| \sum_{j=[nt_1]}^{[nt_2]-1}D^d_{\vec y} (\Delta_\ell\phi)(X_j) \1_{(A_{j,R}^d)^c}(\vec y) \biggr\|_{L^2(\Omega)}^2\,d\vec y\notag\displaybreak[1]\\
    \le{} & \frac1{d!} \sum_{j_1,j_2=[nt_1]}^{[nt_2]-1} \int_{\RR^d} \prod_{r=1}^2 \bigl(\|D^d_{\vec y} (\Delta_\ell\phi)(X_{j_r})\|_{L^2(\Omega)} \1_{(A_{j_r,R}^d)^c}(\vec y)\bigr)\,d\vec y,\label{eq:kernel-derivative-reduction}
\end{align}
where the equality follows from \eqref{eq:WienerItoChaos} and \eqref{e:chaos_expansion}, and we use Cauchy--Schwarz in the last step. By Lemma~\ref{lem:finite-difference-chain-rule}, \eqref{eq:LP-derivatives} and \eqref{eq:LP-input-moments}, we have
\begin{equation} \label{eq:LP-differences}
\begin{split}
    \|D^d_{\vec y} (\Delta_\ell\phi)(X_j)\|_{L^2(\Omega)} &= \biggl\| \int_{[0,1]^d} (\Delta_\ell\phi)^{(d)} \biggl( X_j+ \sum_{i=1}^dt_i\psi_j(x_i) \biggr) \biggl( \prod_{i=1}^d\psi_j(x_i) \biggr) \,d\vec t\biggr\|_{L^2(\Omega)}\\
    &\lesssim 2^{(d-\gamma)\ell} \biggl( \prod_{i=1}^d \Psi_{\alpha,j}(x_i) \biggr).
\end{split}
\end{equation}
Substituting \eqref{eq:LP-differences} into \eqref{eq:kernel-derivative-reduction} and noting that each difference $j_2-j_1$ occurs at most $[nt_2]-[nt_1]$ times, we obtain
\begin{equation} \label{eq:kernel-derivative}
\begin{split}
    & \sum_{j_1,j_2=[nt_1]}^{[nt_2]-1} \int_{\RR^d} \prod_{r=1}^2 \bigl(\|D^d_{\vec y} (\Delta_\ell\phi)(X_{j_r})\|_{L^2(\Omega)} \1_{(A_{j_r,R}^d)^c}(\vec y)\bigr)\,d\vec y\\
    &\lesssim 2^{2(d-\gamma)\ell} \sum_{j_1,j_2=[nt_1]}^{[nt_2]-1} \biggl( \int_{|x-j_1|,|x-j_2|>R} \Psi_{\alpha,j_1}(x) \Psi_{\alpha,j_2}(x) \,dx \biggr)^d\\
    &\lesssim ([nt_2]-[nt_1]) 2^{2(d-\gamma)\ell} \sum_{j\in\ZZ} \biggl( \int_{|x|,|x-j|>R} \Psi_\alpha(x) \Psi_\alpha(x-j) \,dx \biggr)^d.
\end{split}
\end{equation}
For the sum over $|j|\leq 2R$, the Cauchy--Schwarz inequality gives
\begin{equation*}
    \sum_{|j|\le2R} \biggl( \int_{|x|,|x-j|>R} \Psi_\alpha(x) \Psi_\alpha(x-j)\,dx \biggr)^d \lesssim R \biggl( \int_{|x|>R} \Psi_\alpha(x)^2\,dx \biggr)^d \lesssim R^{1-d(2\alpha-1)}.
\end{equation*}
For the sum over $|j|>2R$, we split the integral into $\{|x|\le |j|/2\}$, $\{|x-j|\le |j|/2\}$ and their complement to obtain
\begin{equation*}
\begin{split}
    & \sum_{|j|>2R} \biggl( \int_{|x|,|x-j|>R} \Psi_\alpha(x) \Psi_\alpha(x-j)\,dx \biggr)^d\\
    \lesssim{} & \sum_{j>2R} \biggl[ \biggl( j^{-\alpha} \int_R^{j/2}x^{-\alpha}\,dx \biggr)^d+j^{-d(2\alpha-1)} \biggr]\\
    \lesssim{} & \biggl( \int_R^\infty x^{-\alpha} \biggl( \sum_{j\ge2x}j^{-\alpha d} \biggr)^{\frac1d}\,dx \biggr)^d+R^{1-d(2\alpha-1)} \lesssim R^{1-d(2\alpha-1)},
\end{split}
\end{equation*}
where we use Minkowski's inequality in the second step and $d(2\alpha-1)>1$ in the last inequality. Combining the two bounds with \eqref{eq:kernel-derivative-reduction} and \eqref{eq:kernel-derivative}, we obtain
\begin{equation} \label{eq:kernel-far-L2}
    \sum_{k=d}^\infty \frac1{k!} \int_{\RR^k} \biggl| \sum_{j=[nt_1]}^{[nt_2]-1} \E D^k_{\vec x} (\Delta_\ell\phi)(X_j) \1_{(A_{j,R}^k)^c}(\vec x) \biggr|^2\,d\vec x \lesssim ([nt_2]-[nt_1]) 2^{2(d-\gamma)\ell}R^{1-d(2\alpha-1)}.
\end{equation}
Choose $R=\max\{1,2^{2q_\alpha\ell}\}$. By \eqref{eq:kernel-block-L2}, \eqref{eq:kernel-near-L2} and \eqref{eq:kernel-far-L2}, we have
\begin{align*}
    \biggl\| \sum_{j=[nt_1]}^{[nt_2]-1} \tT^{\ge d} (\Delta_\ell\phi)(X_j) \biggr\|_{L^2(\Omega)}\notag &\lesssim \sqrt{[nt_2]-[nt_1]} \bigl(R 2^{-2\gamma\ell}+2^{2(d-\gamma)\ell}R^{1-d(2\alpha-1)}\bigr)^{\frac{1}{2}}\notag\\
    &\lesssim \sqrt{[nt_2]-[nt_1]}\,2^{-(\gamma-q_\alpha)\ell}.
\end{align*}
This completes the proof of \eqref{eq:sharp-frequency-block}.
\end{proof}

\begin{lem}\label{lem:rough-frequency-block}
Suppose that the assumptions of Theorem~\ref{thm} hold. For every $P\ge2$, we have
\begin{equation}\label{eq:rough-frequency-block}
    \biggl\|\sum_{j=[nt_1]}^{[nt_2]-1}\tT^{\ge d}(\Delta_\ell\phi)(X_j)\biggr\|_{L^P(\Omega)}\lesssim\sqrt{[nt_2]-[nt_1]}\,2^{(d-\gamma)\ell}
\end{equation}
uniformly in $\ell\ge-1$, $0\le t_1<t_2\le1$ and $n\in\NN^+$.
\end{lem}

\begin{proof}
We may assume without loss of generality that $\alpha\ne1$ and $\alpha P/2\ne1$ (by decreasing $\alpha$ and increasing $P$ slightly). By Corollary~\ref{cor:multi_sg_ineq} and the proof of  \eqref{eq:LP-differences} (with $L^P(\Omega)$ in place of $L^2(\Omega)$), we have
\begin{align*}
 \biggl\|\sum_{j=[nt_1]}^{[nt_2]-1}\tT^{\ge d}(\Delta_\ell\phi)(X_j)\biggr\|_{L^P(\Omega)}
 \lesssim &\sum_{v=0}^d\biggl\|\biggl\|\biggl\|\sum_{j=[nt_1]}^{[nt_2]-1}D^d_{\vec x}(\Delta_\ell\phi)(X_j)\biggr\|_{L^P(\Omega)}\biggr\|_{L^P(d\vec y_v)}\biggr\|_{L^2(d\vec z_v)}\\
 \lesssim &2^{(d-\gamma)\ell}\sum_{v=0}^d\biggl\|\biggl\|\sum_{j=[nt_1]}^{[nt_2]-1}\Psi_{\alpha,j}^{\otimes d}\biggr\|_{L^P(d\vec y_v)}\biggr\|_{L^2(d\vec z_v)},
\end{align*}
where $\vec x=(x_1,\ldots,x_d)$, $\vec y_v=(x_1,\ldots,x_v)$ and $\vec z_v=(x_{v+1},\ldots,x_d)$ for $0\le v\le d$. By Minkowski's inequality, we obtain
    \begin{equation} \label{e:mix_bound}
    \begin{split}
        \bigg\|\bigg\| \sum_{j=[nt_1]}^{[nt_2]-1}\Psi_{\alpha,j }^{\otimes d} &(\vec x) \bigg\|_{L^{P}(d\vec y_v)}\bigg\|_{L^2(d\vec z_v)}^2 = \bigg\| \bigg\| \sum_{j_1,j_2=[nt_1]}^{[nt_2]-1}\Psi_{\alpha,j_1 }^{\otimes d}(\vec x)\Psi_{\alpha,j_2 }^{\otimes d}(\vec x)\bigg\|_{L^{\frac{P}{2}}(d\vec y_v)}\bigg\|_{L^1(d\vec z_v)}\\
        &\leq\sum_{j_1,j_2=[nt_1]}^{[nt_2]-1}\|\Psi_{\alpha,j_1 }\Psi_{\alpha,j_2 }\|^v_{L^\frac{P}{2}} \|\Psi_{\alpha,j_1 }\Psi_{\alpha,j_2 }\|_{L^1}^{d-v}\\
        &\lesssim ([nt_2]-[nt_1]) \sum_{j=0}^{\infty}  \langle \Psi_{\frac{\alpha P}{2},0},\Psi_{\frac{\alpha P}{2},j}\rangle ^{\frac{2v}{P}} \langle \Psi_{\alpha,0 }, \Psi_{\alpha,j }\rangle^{d-v}.
    \end{split}
    \end{equation}
    Note that for every $\zeta \in (\frac{1}{2},1)\cup(1,\infty)$, we have 
    \begin{equation*}
        \|\Psi_{\zeta,0}\Psi_{\zeta,j}\|_{L^1(\mathbb R)} \lesssim (1+|j|)^{(1-2\zeta)\vee(-\zeta)}.
    \end{equation*}  
    Applying this bound with $\zeta=\alpha P/2$ and $\zeta=\alpha$, we obtain
\begin{align}\label{e:theta_estimate}
 &\sum_{j=0}^{\infty} \langle \Psi_{\frac{\alpha P}{2},0},\Psi_{\frac{\alpha P}{2},j}\rangle ^{\frac{2v}{P}} \langle \Psi_{\alpha,0 }, \Psi_{\alpha,j } \rangle^{d - v}\notag\\
 \lesssim &\sum_{j=0}^{\infty} (1+j)^{((\frac{2}{P}-2\alpha)\vee(-\alpha))v+((1-2\alpha)\vee(-\alpha))(d-v)} <\infty,
\end{align}
    where the last inequality follows from
    \begin{equation*}
        ((\frac{2}{P}-2\alpha)\vee(-\alpha))v+(d-v)((1-2\alpha)\vee(-\alpha))\leq d((1-2\alpha)\vee(-\alpha))< -1.
    \end{equation*}
Combining \eqref{e:theta_estimate} with \eqref{e:mix_bound}, we obtain the desired estimate \eqref{eq:rough-frequency-block}.
\end{proof}

Now we are ready to prove Theorem~\ref{thm:tightness}.

\begin{proof}[Proof of Theorem~\ref{thm:tightness}]

Decompose $\phi$ into its Littlewood--Paley dyadic blocks $\Delta_\ell\phi (\ell\geq -1)$. Then we have
\begin{equation} \label{e:lp_phi}
    \|\tT^{\ge d}Y_n(t_2)-\tT^{\ge d}Y_n(t_1)\|_{L^p(\Omega)} \leq \sum_{\ell=-1}^\infty \frac1{\sqrt n}\biggl\|\sum_{j=[nt_1]}^{[nt_2]-1}\tT^{\ge d}(\Delta_\ell\phi)(X_j)\biggr\|_{L^p(\Omega)}.
\end{equation}
We interpolate the $L^2(\Omega)$-estimate in Lemma~\ref{lem:sharp-frequency-block} and the $L^P(\Omega)$-estimate in Lemma~\ref{lem:rough-frequency-block} by choosing $p\in(2,P)$ sufficiently close to $2$ such that
\begin{equation}\label{eq:s-range}
    \tau=\frac{\frac12-\frac1p}{\frac12-\frac1P}<\frac{\gamma-q_\alpha}{d-q_\alpha}.
\end{equation}
Then $\frac1p=\frac{1-\tau}{2}+\frac\tau P$. By H\"older's inequality, \eqref{eq:sharp-frequency-block} and \eqref{eq:rough-frequency-block}, we have
\begin{align*}
    &\biggl\|\sum_{j=[nt_1]}^{[nt_2]-1}\tT^{\ge d}(\Delta_\ell\phi)(X_j)\biggr\|_{L^p(\Omega)}\\
    \le & \biggl\|\sum_{j=[nt_1]}^{[nt_2]-1}\tT^{\ge d}(\Delta_\ell\phi)(X_j)\biggr\|_{L^2(\Omega)}^{1-\tau}\biggl\|\sum_{j=[nt_1]}^{[nt_2]-1}\tT^{\ge d}(\Delta_\ell\phi)(X_j)\biggr\|_{L^P(\Omega)}^\tau\notag\\
       \lesssim &\sqrt{[nt_2]-[nt_1]} \cdot 2^{[(1-\tau)(q_\alpha-\gamma)+\tau(d-\gamma)]\ell}.
\end{align*}
Substituting it into \eqref{e:lp_phi} and taking the sum over $\ell$, the conclusion follows from Lemma \ref{lem:tightness}.
\end{proof}

\subsection{Tightness of \TitleEquation{(\mathcal{T}^{\ge d} Z_n)_{n\geq 1}}{(mathcal{T} {ge d} Z n) {n geq 1}}}
\label{sec:tightness_z}
\begin{thm} \label{thm:tightness_z}
    Under the assumptions of Theorem~\ref{thm}, the family $(\mathcal{T}^{\ge d} Z_n)_{n\geq 1}$ is tight in $\mathbf{C}([0,1])$.
\end{thm}

\begin{proof}
According to \cite[Equation~(4.1)]{nualart_breuer_clt}, it suffices to show that, for the $p>2$ chosen in \eqref{eq:s-range}, the inequality
\begin{equation*}
    \|\mathcal{T}^{\ge d} Z_n(t_2) - \mathcal{T}^{\ge d} Z_n(t_1)\|_{L^p(\Omega)} \lesssim \sqrt{t_2-t_1}
\end{equation*}
holds uniformly in $0\leq t_1 < t_2\leq 1$ and $n$.

If $t_2-t_1 > \frac{1}{n}$, then the result follows from \eqref{e:tightness} since
\begin{equation*}
\begin{split}
    \|\mathcal{T}^{\ge d} Z_n(t_2) - \mathcal{T}^{\ge d} Z_n(t_1)\|_{L^p(\Omega)} \lesssim &\|\mathcal{T}^{\ge d} Y_n(t_2) - \mathcal{T}^{\ge d} Y_n(t_1)\|_{L^p(\Omega)}\\
    + \frac{1}{\sqrt{n}} \|\mathcal{T}^{\ge d}&{}\phi(X_{[nt_1]})\|_{L^p(\Omega)} + \frac{1}{\sqrt{n}} \|\mathcal{T}^{\ge d}\phi(X_{[nt_2]})\|_{L^p(\Omega)}
\end{split}
\end{equation*}
and $\|\mathcal{T}^{\ge d}\phi(X_j)\|_{L^p(\Omega)} \lesssim 1$. If $t_2-t_1\leq \frac{1}{n}$ and $(nt_1,nt_2)\cap \mathbb{Z}=\varnothing$, then we have
\begin{equation}\label{eq:Z-single-cell}
    \|\mathcal{T}^{\ge d} Z_n(t_2) - \mathcal{T}^{\ge d} Z_n(t_1)\|_{L^p(\Omega)} = \sqrt{n}(t_2-t_1) \|\mathcal{T}^{\ge d}\phi(X_{[nt_1]})\|_{L^p(\Omega)} \lesssim \sqrt{t_2-t_1}.
\end{equation}
For the remaining case $(nt_1,nt_2)\cap \mathbb{Z}=\{[nt_2]\}$, we have
\begin{equation*}
\begin{aligned}
    &\|\mathcal{T}^{\ge d} Z_n(t_2) - \mathcal{T}^{\ge d} Z_n(t_1)\|_{L^p(\Omega)} \\
    \leq &\|\mathcal{T}^{\ge d} Z_n(t_2) - \mathcal{T}^{\ge d} Z_n(\frac{[nt_2]}{n})\|_{L^p(\Omega)} + \|\mathcal{T}^{\ge d} Z_n(\frac{[nt_2]}{n}) - \mathcal{T}^{\ge d} Z_n(t_1)\|_{L^p(\Omega)}.
\end{aligned}
\end{equation*}
Then the desired bound follows from \eqref{eq:Z-single-cell}. This completes the proof.
\end{proof}

\section{Convergence of finite-dimensional distributions} \label{sec:fdd_convergence}

This section aims to prove the following theorem.
\begin{thm}\label{thm:fdd_convergence}
   Suppose the assumptions in Theorem~\ref{thm} hold. Fix $N\in\NN^+$, \(0<t_1<\cdots<t_N\le1\) and \(b_1, \ldots, b_N \in \mathbb{R}\). Then
        \begin{equation}
            \widetilde{Y}_n:= \sum_{r=1}^N b_r\tT^{\ge d}Y_n(t_r)
        \end{equation}
    converges to $\mu\sum_{r=1}^N b_r B_{t_r}$ in law, where $\mu$ is given in \eqref{e:sigma} and $(B_t)_{t \in [0,1]}$ is a standard Brownian motion. As a consequence, the finite-dimensional distributions of $\mathcal{T}^{\ge d} Y_n$ converge to those of $\mu B$.   
\end{thm}

To prove this theorem, we will truncate the kernel $\psi_j$ to construct the main part, and then use the central limit theorem for a finite-range dependent sequence. For $H\in\NN^+$, we define the truncated sequence by
\begin{equation}\label{eq:kernel-truncation}
    X_{j,H}=I_1(\psi_{j,H}),\qquad \text{where }\psi_{j,H}=\psi_j\1_{[j-H,j+H]}.
\end{equation}
A direct consequence is that $X_{i,H}$ and $X_{j,H}$ are independent if $|i-j|>2H$.

\begin{remark}\label{rem:truncated-frequency-block}
If we replace $X_j$ by $X_{j,H}$, then \eqref{eq:LP-input-moments} and \eqref{eq:sharp-frequency-block} hold with the same proportional constants since $|\psi_{j,H}|\le |\psi_j|$.
\end{remark}

For $t\in[0,1]$, the truncated process is defined by
\begin{equation}\label{eq:truncated-process}
    Y_{n,H}(t)=\frac1{\sqrt n}\sum_{j=0}^{[nt]-1}\phi(X_{j,H}).
\end{equation}

This section is organised as follows.
We prove the central limit theorem of the truncated linear combination $\widetilde{Y}_{n,H} := \sum_{r=1}^Nb_r\tT^{\ge d}Y_{n,H}(t_r)$ for fixed $H$ in Section~\ref{sec:truncated-fdd}. In Section~\ref{sec:kernel-truncation}, we control the $L^2(\Omega)$-error between $\widetilde{Y}_{n,H}$ and $\widetilde{Y}_n$. Then we conclude the proof of Theorem~\ref{thm:fdd_convergence} in Section~\ref{sec:proof_fdd}.

\subsection{Convergence of the truncated sequence}\label{sec:truncated-fdd}

First we provide the definition of $M$-dependent sequence and its corresponding central limit theorem.
\begin{defn}\label{M-dependent}
    Let $M\in\NN$ and $n\in\NN^+$. A sequence of random variables $(U_j)_{j=0}^{n-1}$ is called $M$-dependent if the $\sigma$-algebras generated by $\{U_i\}_{i\in A}$ and $\{U_j\}_{j\in B}$ are independent whenever $A,B\subset\{0,\ldots,n-1\}$ satisfy $|i-j|>M$ for all $i\in A$ and $j\in B$.
\end{defn}

\begin{prop}[{\cite[Theorem~1.1]{Janson2021}}]\label{prop:m-dependent-clt}
Let $M\ge0$ be fixed. For each $n\ge1$, let $(U_{n,j})_{j=0}^{n-1}$ be an $M$-dependent sequence of centred random variables in $L^2(\Omega)$. Assume that
\begin{equation}\label{eq:m-dependent-array-variance}
    \biggl\|\sum_{j=0}^{n-1}U_{n,j}\biggr\|_{L^2(\Omega)}=1,\qquad \sum_{j=0}^{n-1}\E\bigl[|U_{n,j}|^2\1_{\{|U_{n,j}|>\varepsilon\}}\bigr] \overset{n\to \infty}{\longrightarrow}0
\end{equation}
for every $\varepsilon>0$. Then we have
\begin{equation}\label{eq:m-dependent-clt}
    \sum_{j=0}^{n-1}U_{n,j}\xrightarrow[n\to\infty]{\mathrm{law}}N(0,1).
\end{equation}
\end{prop}

\begin{prop}\label{prop:truncated_fdd}
Under the assumptions of Theorem~\ref{thm:fdd_convergence}, fix $H\in\NN^+$. Then
\begin{equation}\label{eq:truncated-linear-limit}
    \widetilde Y_{n,H}:=\frac1{\sqrt n}\sum_{r=1}^Nb_r\sum_{j=0}^{[nt_r]-1}\tT^{\ge d}\phi(X_{j,H})
\end{equation}
converges to $\mu_H\sum_{r=1}^Nb_rB_{t_r}$ in law, where $\mu_H\ge0$ is defined by
\begin{equation}\label{eq:smooth-variance}
    \mu_H^2 =\sum_{j=-2H}^{2H}\E[\tT^{\ge d}\phi(X_{0,H})\tT^{\ge d}\phi(X_{j,H})].
\end{equation}
\end{prop}

\begin{proof}
By the stationarity and $2H$-dependence of $(X_{j,H})_{j\in\Z}$, we have
\begin{align}
 \|\widetilde Y_{n,H}\|_{L^2(\Omega)}^2
 ={} &\sum_{j=-2H}^{2H}\E[\tT^{\ge d}\phi(X_{0,H})\tT^{\ge d}\phi(X_{j,H})]\sum_{r_1,r_2=1}^N\frac{b_{r_1}b_{r_2}}{n}\sum_{i=0}^{[nt_{r_1}]-1}\1_{\{0\le i+j<[nt_{r_2}]\}}\notag\\
 \overset{n\to\infty}{\longrightarrow} &\mu_H^2\sum_{r_1,r_2=1}^Nb_{r_1}b_{r_2}(t_{r_1}\wedge t_{r_2})=\biggl\|\mu_H\sum_{r=1}^Nb_rB_{t_r}\biggr\|_{L^2(\Omega)}^2.\label{eq:truncated-linear-variance}
\end{align}
If $\mu_H^2\sum_{r_1,r_2=1}^Nb_{r_1}b_{r_2}(t_{r_1}\wedge t_{r_2})=0$, then \eqref{eq:truncated-linear-variance} implies
\begin{equation*}
    \widetilde Y_{n,H}\xrightarrow[n\to\infty]{L^2(\Omega)}0.
\end{equation*}
Otherwise, \eqref{eq:truncated-linear-variance} implies that $\|\widetilde Y_{n,H}\|_{L^2(\Omega)}>0$ for all sufficiently large $n$. We set
\begin{equation*}
    U_{n,j}=\frac{\bigl(\sum_{r=1}^Nb_r\1_{\{0\le j<[nt_r]\}}\bigr)\tT^{\ge d}\phi(X_{j,H})}{\sqrt n\,\|\widetilde Y_{n,H}\|_{L^2(\Omega)}},\qquad0\le j\le n-1.
\end{equation*}
Then $(U_{n,j})_{j=0}^{n-1}$ is centred and $2H$-dependent, and its sum $\widetilde Y_{n,H}/\|\widetilde Y_{n,H}\|_{L^2(\Omega)}$ has unit variance.
By \eqref{eq:truncated-linear-variance}, there exists $C>0$ such that for all sufficiently large $n\in\NN^+$, we have
\begin{align*}
    |U_{n,j}|
    \le\frac C{\sqrt n}|\tT^{\ge d}\phi(X_{j,H})|.
\end{align*}
Consequently, for every $\varepsilon>0$, we have
\begin{align*}
    &\sum_{j=0}^{n-1} \E\bigl[|U_{n,j}|^2\1_{\{|U_{n,j}|>\varepsilon\}}\bigr]
    \le C^2\E\Bigl[|\tT^{\ge d}\phi(X_{0,H})|^2\1_{\{|\tT^{\ge d}\phi(X_{0,H})|>\frac{\varepsilon\sqrt n}{C}\}}\Bigr]\overset{n\to \infty}{\longrightarrow} 0,
\end{align*}
where the last limit follows from the dominated convergence theorem since
\begin{equation*}
    \E|\tT^{\ge d}\phi(X_{0,H})|^2\le\E|\phi(X_{0,H})|^2\lesssim\E(1+X_{0,H}^2)^M<\infty.
\end{equation*}
By Proposition~\ref{prop:m-dependent-clt}, we obtain
\begin{equation*}
    \frac{\widetilde Y_{n,H}}{\|\widetilde Y_{n,H}\|_{L^2(\Omega)}}\xrightarrow[n\to\infty]{\mathrm{law}}N(0,1).
\end{equation*}
Combining it with \eqref{eq:truncated-linear-variance}, we complete the proof.
\end{proof}

\subsection{Estimate of remainder}\label{sec:kernel-truncation}

Recall that we use the truncated process $Y_{n,H}$ given in \eqref{eq:truncated-process} to approximate $Y_n$. The following proposition controls the error of the remainder.
\begin{prop}\label{prop:truncated_block_Lp}
Suppose that the assumptions of Theorem~\ref{thm} hold. For every $t\in[0,1]$, there exists $\kappa>0$ such that
\begin{equation}\label{eq:smooth-kernel-approximation}
    \sup_{n\ge1} \|\tT^{\ge d} (Y_n(t)-Y_{n,H}(t)) \|_{L^2(\Omega)} \lesssim H^{-\kappa}
\end{equation}
holds uniformly in $H\in\NN^+$.
\end{prop}

\begin{proof}
Decompose $\phi$ into its Littlewood--Paley dyadic blocks $\Delta_\ell\phi (\ell\geq -1)$ and set $F_{j,\ell}=\tT^{\ge d}((\Delta_\ell\phi)(X_j)-(\Delta_\ell\phi)(X_{j,H}))$. Then we have
\begin{equation}\label{eq:kernel-error-frequency}
    \|\tT^{\ge d}(Y_n(t)-Y_{n,H}(t))\|_{L^2(\Omega)}\le\frac1{\sqrt n}\sum_{\ell=-1}^\infty\biggl\|\sum_{j=0}^{[nt]-1}F_{j,\ell}\biggr\|_{L^2(\Omega)}.
\end{equation}
By Corollary~\ref{cor:multi_sg_ineq} with $p=2$, we get
\begin{equation}\label{eq:kernel-error-reduction}
    \biggl\| \sum_{j=0}^{[nt]-1} F_{j,\ell} \biggr\|_{L^2(\Omega)} \lesssim \biggl\| \sum_{j=0}^{[nt]-1} \| D^d_{\vec x} F_{j,\ell} \|_{L^2(\Omega)} \biggr\|_{L^2(\RR^d)},
\end{equation}
where $\vec x=(x_1,\ldots,x_d)$. Similar to \eqref{eq:LP-differences}, we have
\begin{align*}
    D^d_{\vec x}F_{j,\ell}
    = &\biggl(\prod_{i=1}^d\psi_j(x_i)-\prod_{i=1}^d\psi_{j,H}(x_i)\biggr)\int_{[0,1]^d}(\Delta_\ell\phi)^{(d)}\biggl(X_j+\sum_{i=1}^dt_i\psi_j(x_i)\biggr)\,d\vec t\notag\\
    &+\biggl(\prod_{i=1}^d\psi_{j,H}(x_i)\biggr)\int_{[0,1]^d}\int_{X_{j,H}}^{X_j}(\Delta_\ell\phi)^{(d+1)}\biggl(r+\sum_{i=1}^dt_i\psi_j(x_i)\biggr)\,dr\,d\vec t.
\end{align*}
Note that by \eqref{as:psi}, we have
\begin{align*}
    \biggl|\prod_{i=1}^d\psi_j(x_i)-\prod_{i=1}^d\psi_{j,H}(x_i)\biggr|
    \le &\Bigl(\prod_{i=1}^d |\psi_j(x_i)| \Bigr) \Bigl( \sum_{i=1}^d \1_{\{|x_i-j|>H\}} \Bigr) \lesssim H^{-\delta} \prod_{i=1}^d \Psi_{\alpha-\delta,j}(x_i) 
\end{align*}
for every sufficiently small $\delta>0$. Then by \eqref{eq:LP-derivatives} and  \eqref{eq:LP-input-moments}, we get
\begin{align*}
    \|D^d_{\vec x}F_{j,\ell}\|_{L^2(\Omega)} \lesssim &\Bigl( \prod_{i=1}^d \Psi_{\alpha-\delta,j}(x_i) \Bigr) \Bigl( 2^{(d-\gamma)\ell} H^{-\delta} + 2^{(d+1-\gamma)\ell} \|X_j-X_{j,H}\|_{L^4(\Omega)} \Bigr)\\
    \lesssim &H^{-\delta}2^{(d+1-\gamma)\ell}\biggl(\prod_{i=1}^d\Psi_{\alpha-\delta,j}(x_i)\biggr),
\end{align*}
where the last inequality follows from
\begin{align}\label{eq:kernel-error-input-tail}
    \|X_j-X_{j,H}\|_{L^4(\Omega)} = \| I_1(\psi \1_{\{|\cdot|>H\}}) \|_{L^4(\Omega)} \lesssim H^{-(\alpha-\frac12)}.
\end{align}
Substituting this bound into \eqref{eq:kernel-error-reduction}, and using \eqref{e:mix_bound} and \eqref{e:theta_estimate} with $P=2$ and $v=0$, we obtain
\begin{align}\label{eq:kernel-error-block}
    &\frac1{\sqrt n}\biggl\|\sum_{j=0}^{[nt]-1}F_{j,\ell}\biggr\|_{L^2(\Omega)}
    \lesssim &\frac{H^{-\delta}2^{(d+1-\gamma)\ell}}{\sqrt n}\biggl\|\sum_{j=0}^{[nt]-1}\Psi_{\alpha-\delta,j}^{\otimes d}\biggr\|_{L^2(\RR^d)}
    \lesssim H^{-\delta} 2^{(d+1-\gamma)\ell}.
\end{align}
On the other hand, the uniform bound in Lemma~\ref{lem:sharp-frequency-block} implies that
\begin{equation} \label{eq:kernel-error-uniform}
    \frac1{\sqrt n}\biggl\|\sum_{j=0}^{[nt]-1}F_{j,\ell}\biggr\|_{L^2(\Omega)}
    \lesssim 2^{-(\gamma-q_\alpha)\ell}.
\end{equation}
Combining \eqref{eq:kernel-error-frequency}, \eqref{eq:kernel-error-uniform} and \eqref{eq:kernel-error-block}, we have
\begin{align*}
    \|\tT^{\ge d}(Y_n(t)-Y_{n,H}(t))\|_{L^2(\Omega)}
    \lesssim &\sum_{\ell=-1}^\infty\min\bigl\{2^{-(\gamma-q_\alpha)\ell},2^{(d+1-\gamma)\ell} H^{-\delta}\bigr\} \lesssim H^{-\kappa}
\end{align*}
for some $\kappa>0$. This completes the proof.
\end{proof}

\subsection{Proof of Theorem~\ref{thm:fdd_convergence}} \label{sec:proof_fdd}

\begin{proof}[Proof of Theorem~\ref{thm:fdd_convergence}]
We first show the limit in \eqref{e:sigma} exists. Recall that we approximate the process $Y_n$ by the truncated process $Y_{n,H}$. For $n,n'\ge1$, by \eqref{eq:smooth-kernel-approximation}, we have
\begin{align*}
    &\bigl|\|\tT^{\ge d}Y_n(1)\|_{L^2(\Omega)}-\|\tT^{\ge d}Y_{n'}(1)\|_{L^2(\Omega)}\bigr|\\
    \le &2\sup_{k\ge1}\|\tT^{\ge d}(Y_k(1)-Y_{k,H}(1))\|_{L^2(\Omega)}+\bigl|\|\tT^{\ge d}Y_{n,H}(1)\|_{L^2(\Omega)}-\|\tT^{\ge d}Y_{n',H}(1)\|_{L^2(\Omega)}\bigr|.
\end{align*}
This implies that $(\|\tT^{\ge d}Y_n(1)\|_{L^2(\Omega)})_{n\ge1}$ is Cauchy since $\|\tT^{\ge d}Y_{n,H}(1)\|_{L^2(\Omega)}\to\mu_H$ as $n\to\infty$. Moreover, its limit $\mu\ge0$ satisfies \eqref{e:sigma}. As $H\to\infty$, we have
\begin{equation*}
    |\mu_H-\mu| \le \sup_{n\ge1}\|\tT^{\ge d}(Y_n(1)-Y_{n,H}(1))\|_{L^2(\Omega)}\to 0.
\end{equation*}
We complete the proof by combining it with Propositions~\ref{prop:truncated_fdd} and \ref{prop:truncated_block_Lp}.
\end{proof}

As a consequence, the convergence of finite-dimensional distributions also holds for $\mathcal{T}^{\ge d} Z_n$.
\begin{cor} \label{cor:z_fdd}
    The same result as in Theorem~\ref{thm:fdd_convergence} holds for $\mathcal{T}^{\ge d} Z_n$.
\end{cor}

\begin{proof}
    Since $Z_n(t) = Y_n(t) + \frac{nt-[nt]}{\sqrt{n}}\,\phi(X_{[nt]})$, the result follows from
    \begin{equation*}
        \mathbf{E}\left(\frac{nt-[nt]}{\sqrt{n}}\,\mathcal{T}^{\ge d}\phi(X_{[nt]})\right)^2\leq \frac{1}{n}\mathbf{E}\left(\mathcal{T}^{\ge d}\phi(X_0)\right)^2 \overset{n\to\infty}{\longrightarrow} 0.
    \end{equation*}
\end{proof}

\section{Sharpness of the regularity condition}\label{sec:sharpness}

Fix an integer $M$ such that $M>100\alpha d$. Define $h_m=2^{\lfloor q_\alpha M^m\rfloor}$ for $m\ge1$. The function $\phi$ is defined by
\begin{equation*}
    \phi(x)=\sum_{\ell=1}^\infty \phi_\ell(x), \qquad \phi_\ell(x)=\frac{\cos(2\pi2^{M^\ell}x)}{h_\ell}.
\end{equation*}
We set $\psi$ as a piecewise constant function. For $s\ge1$ and $j\in\ZZ$, set $I_{j,s}=[j+h_s^2,j+2h_s^2)$. Then $\psi_j$ is defined by
\begin{equation*}
    \psi_j(x)=\sum_{s=1}^\infty \psi_{j,s}(x), \qquad \psi_{j,s}(x) = \frac{2^{-M^s}}{h_s}\1_{I_{j,s}}(x)
    .
\end{equation*}
In the proof below, we choose $N=h_m^8$ and $H=h_{m+1}^2$. We show that $\tT^{\ge d}Y_N(1)$ is not tight. We can approximate $\tT^{\ge d}Y_N(1)$ by the truncated sequence $\tT^{\ge d}Y_{N,H}(1)$, where $Y_{N,H}$ is defined by \eqref{eq:truncated-process}. We show that $\tT^{\ge d}Y_{N,H}(1)$ is not tight using the divergence of its second moment, a fourth-moment bound and Paley--Zygmund's inequality. 

Recall the construction of the stationary sequence $X_j$. We get
\begin{equation} \label{e:cos_integer}
    \cos (2\pi 2^{M^\ell} X_{j,H}) = \cos(2\pi \sum_{s=1}^m I_1(2^{M^\ell} \psi_{j,s})) = \cos(2\pi \sum_{s=\ell}^m I_1(2^{M^\ell} \psi_{j,s}))
\end{equation}
where the last equality holds since $2^{M^\ell} \psi_{j,s}$ takes integer values for $s<\ell$. Therefore, we have
\begin{equation} \label{e:endpoint_frequency_reduction}
    \tT^{\ge d}Y_{N,H}(1)
    = \frac1{\sqrt N}\sum_{j=0}^{N-1}\sum_{\ell=1}^m\tT^{\ge d}\phi_\ell(X_{j,H}).
\end{equation}

\begin{rmk}
We give a heuristic explanation for the divergence of
\begin{equation*}
    \sigma_m^2=\E\bigl|\tT^{\ge d}Y_{N,H}(1)\bigr|^2.
\end{equation*}
According to \eqref{e:cos_integer}, we have
\begin{equation*}
    \tT^{\ge d}\phi_\ell(X_{j,H}) = \frac1{h_\ell} \tT^{\ge d} \cos\Bigl(2\pi2^{M^\ell}\sum_{s=\ell} ^m  I_1(\psi_{j,s}) \Bigr).
\end{equation*}
Since $\psi$ decays rapidly, the main contribution comes from $\psi_{j,\ell}$. Then we can replace $\phi_\ell(X_{j,H})$ by $\phi_\ell(I_1(\psi_{j,\ell}))$. If we neglect correlations between different frequencies, then we formally have
\begin{equation*}
    \sigma_m^2\approx\frac1N\sum_{\ell=1}^m\sum_{j,k=0}^{N-1}\E\bigl[\tT^{\ge d}\phi_\ell(I_1(\psi_{j,\ell}))\tT^{\ge d}\phi_\ell(I_1(\psi_{k,\ell}))\bigr].
\end{equation*}
Note that $I_1(\psi_{j,\ell})$ and $I_1(\psi_{k,\ell})$ are independent for $|j-k|\ge h_\ell^2$. Thus, we can heuristically write
\begin{align*}
    &\frac1N\sum_{j,k=0}^{N-1}\E\bigl[\tT^{\ge d}\phi_\ell(I_1(\psi_{j,\ell}))\tT^{\ge d}\phi_\ell(I_1(\psi_{k,\ell}))\bigr]\\
    \approx &\frac1N\E\bigl|\tT^{\ge d}\phi_\ell(I_1(\psi_{0,\ell}))\bigr|^2\,(Nh_\ell^2)\sim h_\ell^{-2}h_\ell^2,
\end{align*}
where $h_\ell^{-2}$ comes from the coefficient in $\phi_\ell$. Summing over $\ell=1,\ldots,m$ yields $\sigma_m^2\gtrsim m$.

\end{rmk}

\begin{lem}\label{lem:endpoint_frequency_reduction}
For $m,\ell\in\NN^+$, $j\in\ZZ$,  $k\ge0$ and $\vec x\in\RR^k$, we have
\begin{align}\label{e:endpoint_exponential_difference}
    D^k_{\vec x}e^{2\pi{\rm i}2^{M^\ell}X_{j,H}} = \exp\Bigl(2\pi{\rm i}\sum_{s=\ell}^m I_1(2^{M^\ell}\psi_{j,s})\Bigr) \prod_{i=1}^k\biggl(\sum_{s=\ell}^m\Bigl(e^{2\pi{\rm i}2^{M^\ell}\psi_{j,s}(x_i)}-1\Bigr)\biggr),
\end{align}
where $\vec x=(x_1,\ldots,x_k)$. 
\end{lem}
\begin{proof}
Similar to \eqref{e:cos_integer}, we have
\begin{align*}
    e^{2\pi{\rm i}2^{M^\ell}X_{j,H}}= \exp\Bigl(2\pi{\rm i}\sum_{s=\ell}^m I_1(2^{M^\ell}\psi_{j,s})\Bigr).
\end{align*}
Then the desired result follows from \eqref{eq:difference-definition}. 
\end{proof}

\begin{lem}\label{lem:endpoint_variance}
We have $\sigma_m^2\gtrsim m$ uniformly in $m\ge1$.
\end{lem}
\begin{proof}
By \eqref{e:endpoint_frequency_reduction}, \eqref{e:chaos_expansion} and taking only the $2d$-th chaos, we obtain the lower bound
\begin{equation}\label{e:endpoint_variance_reduction}
    \sigma_m^2 \ge \frac1{(2d)!N}\int_{\RR^{2d}}\biggl|\sum_{j=0}^{N-1}\sum_{\ell=1}^m(-1)^d\E D^{2d}_{\vec x}\phi_\ell(X_{j,H})\biggr|^2\,d\vec x,
\end{equation}
where $\vec x=(x_1,\ldots,x_{2d})$. Here, we put $(-1)^d$ into the second moment since we will bound $(-1)^d\E D^{2d}_{\vec x}\phi_\ell(X_{j,H})$ from below. By \eqref{e:endpoint_exponential_difference} and $e^{{\rm i}t}-1=2{\rm i}e^{\frac{{\rm i}t}{2}}\sin\frac t2$, we have
\begin{equation}\label{e:endpoint_kernel_formula}
\begin{aligned}
    &(-1)^d\E D^{2d}_{\vec x}\phi_\ell(X_{j,H})\\
    = &\frac{(-1)^d}{h_\ell}\operatorname{Re}\biggl[\E e^{2\pi{\rm i}\sum_{s=\ell}^m I_1(2^{M^\ell}\psi_{j,s})}\prod_{i=1}^{2d}\biggl(\sum_{s=\ell}^m\bigl(e^{2\pi{\rm i}2^{M^\ell}\psi_{j,s}(x_i)}-1\bigr)\biggr)\biggr]\\
    = &\frac{2^{2d}}{h_\ell}\operatorname{Re} (z) \cdot\prod_{i=1}^{2d}\biggl(\sum_{s=\ell}^m\sin\bigl(\pi2^{M^\ell}\psi_{j,s}(x_i)\bigr)\biggr),
\end{aligned}
\end{equation}
where
\begin{equation*}
\begin{aligned}
    z= &e^{\pi{\rm i}2^{M^\ell}\sum_{i=1}^{2d}\sum_{s=\ell}^m\psi_{j,s}(x_i)}\E e^{2\pi{\rm i}\sum_{s=\ell}^m I_1(2^{M^\ell}\psi_{j,s})}\\
    = &\exp\biggl(\pi{\rm i}2^{M^\ell}\sum_{i=1}^{2d}\sum_{s=\ell}^m\psi_{j,s}(x_i)+\sum_{s=\ell}^m h_s^2\biggl(e^{\frac{2\pi{\rm i}2^{M^\ell-M^s}}{h_s}}-1-\frac{2\pi{\rm i}2^{M^\ell-M^s}}{h_s}\biggr)\biggr).
\end{aligned}
\end{equation*}
We have
\begin{equation*}
\begin{aligned}
    \operatorname{Re}z
    = &\exp\biggl(\sum_{s=\ell}^m h_s^2\Bigl(\cos\frac{2\pi2^{M^\ell-M^s}}{h_s}-1\Bigr)\biggr)\\
    &\quad\times\cos\biggl(\pi2^{M^\ell}\sum_{i=1}^{2d}\sum_{s=\ell}^m\psi_{j,s}(x_i)+\sum_{s=\ell}^m h_s^2\Bigl(\sin\frac{2\pi2^{M^\ell-M^s}}{h_s}-\frac{2\pi2^{M^\ell-M^s}}{h_s}\Bigr)\biggr).
\end{aligned}
\end{equation*}
Using $|\sin t-t|\le\frac{|t|^3}{6}$, we obtain
\begin{equation*}
\begin{aligned}
    &\biggl|\pi2^{M^\ell}\sum_{i=1}^{2d}\sum_{s=\ell}^m\psi_{j,s}(x_i)+\sum_{s=\ell}^m h_s^2\biggl(\sin\frac{2\pi2^{M^\ell-M^s}}{h_s}-\frac{2\pi2^{M^\ell-M^s}}{h_s}\biggr)\biggr|\\
    \le &2\pi d\sum_{s=\ell}^m\frac{2^{M^\ell-M^s}}{h_s}+\frac{(2\pi)^3}{6}\sum_{s=\ell}^m\frac{2^{3(M^\ell-M^s)}}{h_s} = \oO(h_\ell^{-1}).
\end{aligned}
\end{equation*}
Thus, using $\cos t-1\ge-\frac{t^2}{2}$, we have
\begin{equation}\label{e:endpoint_characteristic}
    \operatorname{Re}z\gtrsim \exp\Bigl(-2\pi^2\sum_{s=\ell}^m2^{2(M^\ell-M^s)}\Bigr)\gtrsim1.
\end{equation}
For the second factor in \eqref{e:endpoint_kernel_formula}, we have
\begin{equation}\label{e:endpoint_sine_product}
    \prod_{i=1}^{2d}\biggl( \sum_{s=\ell}^m \sin\bigl(\pi2^{M^\ell}\psi_{j,s}(x_i)\bigr) \biggr)
    \gtrsim \frac{2^{2d}}{h_\ell^{2d}}\biggl(\prod_{i=1}^{2d}\1_{I_{j,\ell}}(x_i)\biggr)\ge0,
\end{equation}
where we use $\sin t\ge\frac{2t}{\pi}$ for small $t>0$.
By \eqref{e:endpoint_kernel_formula}, \eqref{e:endpoint_characteristic} and \eqref{e:endpoint_sine_product}, for every $\vec x\in\RR^{2d}$, we have
\begin{equation*}
    (-1)^d\E D^{2d}_{\vec x}\phi_\ell(X_{j,H}) \gtrsim \frac{1}{h_\ell^{2d+1}}\biggl(\prod_{i=1}^{2d}\1_{I_{j,\ell}}(x_i)\biggr)\ge0.
\end{equation*}
Substituting this bound into \eqref{e:endpoint_variance_reduction}, we have
\begin{equation}\label{e:endpoint_variance_blocks}
\begin{aligned}
    \sigma_m^2
    \gtrsim &\frac1{N}\int_{\RR^{2d}}\biggl(\sum_{\ell=1}^m\frac1{h_\ell^{2d+1}}\sum_{j=0}^{N-1}\prod_{i=1}^{2d}\1_{I_{j,\ell}}(x_i)\biggr)^2\,d\vec x\\
    \ge &\sum_{\ell=1}^m\frac1{Nh_\ell^{4d+2}}\int_{\RR^{2d}}\biggl(\sum_{j=0}^{N-1}\prod_{i=1}^{2d}\1_{I_{j,\ell}}(x_i)\biggr)^2\,d\vec x\\
    = &\sum_{\ell=1}^m \frac1{Nh_\ell^{4d+2}} \sum_{j,k=0}^{N-1} \biggl(\int_\RR\1_{I_{j,\ell}}(x)\1_{I_{k,\ell}}(x)\,dx\biggr)^{2d}.
\end{aligned}
\end{equation}
Note that 
\begin{equation}\label{e:endpoint_interval_overlap}
    \int_\RR\1_{I_{j,\ell}}(x)\1_{I_{k,\ell}}(x)\,dx=(h_\ell^2-|j-k|)_+\ge\frac{h_\ell^2}{2}\1_{\{|j-k|\le\frac{h_\ell^2}{2}\}},
\end{equation}
where $a_+=\max\{a,0\}$ for $a\in\RR$.
Therefore, we have
\begin{equation}\label{e:endpoint_variance_lower}
\begin{aligned}
    \sigma_m^2
    \gtrsim \sum_{\ell=1}^m\frac1{Nh_\ell^2} \sum_{j,k=0}^{N-1} \1_{\{|j-k|\le\frac{h_\ell^2}{2}\}} \gtrsim  \sum_{\ell=1}^m \frac1{Nh_\ell^2} N h_\ell^2 \gtrsim m.
\end{aligned}
\end{equation}
\end{proof}

The following lemma provides a fourth-moment estimate.
\begin{lem}\label{lem:endpoint_fourth_moment}
We have
\begin{equation}\label{e:endpoint_fourth_moment_bounds}
    \E\bigl|\tT^{\ge d}Y_{N,H}(1)\bigr|^4 - 3\sigma_m^4 \lesssim 1,
\end{equation}
uniformly in $m \ge 1$.
\end{lem}

\begin{proof}

Expanding the fourth moment and $\sigma_m^4$, we obtain
\begin{equation}\label{e:endpoint_fourth_expansion}
\begin{aligned}
    \E\bigl|\tT^{\ge d}Y_{N,H}(1)\bigr|^4 &- 3\sigma_m^4 =  \frac1{N^2}\sum_{j_1,j_2,j_3,j_4=0}^{N-1}\biggl(\E\biggl[\prod_{r=1}^4\tT^{\ge d}\phi(X_{j_r,H})\biggr]\\
    &- \sum_{i=2}^4\E[\tT^{\ge d}\phi(X_{j_1,H})\tT^{\ge d}\phi(X_{j_i,H})]\E\biggl[\prod_{\substack{k=2\\ k \ne i}}^4\tT^{\ge d}\phi(X_{j_k,H})\biggr]\biggr).
\end{aligned}
\end{equation}
Recall that the sequence $(\tT^{\ge d}\phi(X_{j,H}))_{j \in \ZZ}$ is centered, stationary and $2h_m^2$-dependent. By symmetry, it suffices to consider the case $j_1 \le j_2 \le j_3 \le j_4$. If $j_2 - j_1 > 2h_m^2$ or $j_4 - j_3 > 2h_m^2$, then the independence implies that all four terms in the parenthesis vanish. If $j_3 - j_2 > 2h_m^2$, the expression in the parenthesis also vanishes since
\begin{equation*}
    \E\biggl[\prod_{r=1}^4\tT^{\ge d}\phi(X_{j_r,H})\biggr] = \E[\tT^{\ge d}\phi(X_{j_1,H}) \tT^{\ge d}\phi(X_{j_2,H})] \E[\tT^{\ge d}\phi(X_{j_3,H}) \tT^{\ge d}\phi(X_{j_4,H})]
\end{equation*}
and
\begin{equation}\notag
    \E[\tT^{\ge d}\phi(X_{j_1,H})\tT^{\ge d}\phi(X_{j_i,H})]\E\biggl[\prod_{\substack{k=2\\ k \ne i}}^4\tT^{\ge d}\phi(X_{j_k,H})\biggr] = 0
\end{equation}
for $i=3,4$. Consequently, the remaining terms in the summation over $0\leq j_1\leq j_2\leq j_3\leq j_4 \leq N-1$ must satisfy $|j_{i+1} - j_i|\leq 2h_m^2 (i=1,2,3)$. There are at most $N(2h_m^2+1)^3$ terms. Using H\"older's inequality to get the termwise bound on these remaining terms, we obtain
\begin{equation}\label{e:endpoint_block_fourth_moment}
\begin{aligned}
    &\bigl|\E|\tT^{\ge d}Y_{N,H}(1)|^4 - 3\sigma_m^4\bigr|\\
    \leq& \frac{4!N(2h_m^2 + 1)^3}{N^2}\biggl(\prod_{r=1}^4\|\tT^{\ge d}\phi(X_{j_r,H})\|_{L^4(\Omega)} + 3\prod_{r=1}^4 \|\tT^{\ge d} \phi(X_{j_r,H}) \|_{L^2(\Omega)} \biggr)\\
    \lesssim & \frac{\|\tT^{\ge d}\phi(X_{0,H})\|_{L^4(\Omega)}^4}{h_m^2}.
\end{aligned}
\end{equation}

Now we estimate $\|\tT^{\ge d}\phi(X_{0,H})\|_{L^4(\Omega)}$. By Corollary~\ref{cor:multi_sg_ineq}, we have
\begin{align}\label{e:endpoint_fourth_moment_reduction}
    \|\tT^{\ge d}\phi(X_{0,H})\|_{L^4(\Omega)}
    \le &\sum_{\ell=1}^m\|\tT^{\ge d}\phi_\ell(X_{0,H})\|_{L^4(\Omega)}\notag\\
    \lesssim &\sum_{\ell=1}^m\sum_{v=0}^d\Bigl\|\Bigl\|\|D^d_{\vec x}\phi_\ell(X_{0,H})\|_{L^4(\Omega)}\Bigr\|_{L^4(d\vec y_v)}\Bigr\|_{L^2(d\vec z_v)},
\end{align}
where $\vec x=(x_1,\ldots,x_d)$, $\vec y_v=(x_1,\ldots,x_v)$ and $\vec z_v=(x_{v+1},\ldots,x_d)$ for $0\le v\le d$.
For $1\le\ell\le m$, by \eqref{e:endpoint_exponential_difference} and $|e^{{\rm i}t}-1|\le|t|$, we have
\begin{align}\label{e:endpoint_difference_absolute}
    |D^d_{\vec x}\phi_\ell(X_{0,H})|
    \le &\frac1{h_\ell}\prod_{i=1}^d\biggl(\sum_{s=\ell}^m\Bigl|e^{2\pi{\rm i}2^{M^\ell}\psi_{0,s}(x_i)}-1\Bigr|\biggr)\notag\\
    \lesssim &\frac{1}{h_\ell}\prod_{i=1}^d\biggl(\sum_{s=\ell}^m2^{M^\ell}\psi_{0,s}(x_i)\biggr)\notag\\
    = &\frac1{h_\ell}\prod_{i=1}^d\biggl(\sum_{s=\ell}^m\frac{2^{M^\ell-M^s}}{h_s}\1_{I_{0,s}}(x_i)\biggr).
\end{align}
Therefore, we have
\begin{align}\label{e:endpoint_fourth_moment}
    \|\tT^{\ge d}\phi(X_{0,H})\|_{L^4(\Omega)}    \lesssim &\sum_{\ell=1}^m\frac1{h_\ell}\sum_{v=0}^d\biggl[\biggl(\sum_{s=\ell}^m\biggl(\frac{2^{M^\ell-M^s}}{h_s}\biggr)^4h_s^2\biggr)^{\frac v4}    \biggl(\sum_{s=\ell}^m\biggl(\frac{2^{M^\ell-M^s}}{h_s}\biggr)^2h_s^2\biggr)^{\frac{d-v}{2}}\biggr]\notag\\
    \le &\sum_{\ell=1}^m\frac1{h_\ell}\sum_{v=0}^d\biggl[h_\ell^{-\frac v2}\biggl(\sum_{s=\ell}^m2^{-4(s-\ell)}\biggr)^{\frac v4}    \biggl(\sum_{s=\ell}^m2^{-2(s-\ell)}\biggr)^{\frac{d-v}{2}}\biggr]
    \lesssim1.
\end{align}
Substituting \eqref{e:endpoint_fourth_moment} into \eqref{e:endpoint_block_fourth_moment}, we obtain \eqref{e:endpoint_fourth_moment_bounds}.
\end{proof}

\begin{proof}[Proof of Proposition~\ref{prop:sharpness}]

First we verify that $\psi$ satisfies \eqref{as:psi} and $\phi\in\cC_0^{q_\alpha}$. For $m\ge1$, the intervals $I_{0,m}$ are pairwise disjoint. For $x\in I_{0,m}$, we have $|\psi(x)|=\frac{2^{-M^m}}{h_m}\le h_m^{-2\alpha}\lesssim(1+|x|)^{-\alpha}$.
, we have $\|\phi\|_{L^\infty(\RR)}\le \sum_{m=1}^\infty\frac1{h_m}<\infty$.
Recall the definition of $\Delta_\ell \phi$ given in \eqref{eq:LP-definition}, for $\ell\ge0$, we obtain
\begin{align*}
    \|\Delta_\ell\phi\|_{L^\infty(\RR)}
    \lesssim \sum_{m=1}^\infty\biggl(\frac1{h_m}\1_{\{\frac34\le2^{M^m-\ell}\le\frac83\}}\biggr)
    \lesssim 2^{-q_\alpha\ell}.
\end{align*}
Therefore, by \eqref{eq:weighted-LP-norm}, we have $\phi\in\cC_0^{q_\alpha}$.

We prove that $(\tT^{\ge d}Y_n(1))_{n\ge1}$ is not tight. First we estimate the truncation error. By \eqref{e:endpoint_frequency_reduction}, \eqref{e:chaos_expansion} and \eqref{eq:kernel-error-input-tail}, we have
\begin{align}\label{e:endpoint_total_approximation}
    \bigl\|\tT^{\ge d}Y_N(1)-\tT^{\ge d}Y_{N,H}(1)\bigr\|_{L^2(\Omega)} \le &\sqrt N\biggl[\sum_{\ell>m}\frac1{h_\ell}+\sum_{\ell=1}^m\|\phi_\ell(X_0)-\phi_\ell(X_{0,H})\|_{L^2(\Omega)}\biggr]\notag\\
    \le &\sqrt N\biggl[\sum_{\ell>m}\frac1{h_\ell}+2\pi\biggl(\sum_{\ell=1}^m\frac{2^{M^\ell}}{h_\ell}\biggr)\|X_0-X_{0,H}\|_{L^2(\Omega)}\biggr]\notag\\
    \lesssim &\sqrt N\biggl[\sum_{\ell>m}\frac1{h_\ell}+H^{\frac12-\alpha}\biggl(\sum_{\ell=1}^m\frac{2^{M^\ell}}{h_\ell}\biggr)\biggr]\notag\\
    \lesssim &2^{(4-M)q_\alpha M^m}+2^{(4q_\alpha+1-M)M^m}\longrightarrow0.
\end{align}
By Lemmas~\ref{lem:endpoint_variance}
and~\ref{lem:endpoint_fourth_moment}, and Paley--Zygmund's inequality, we have
\begin{equation*}
    \mathbf P\Bigl( |\tT^{\ge d}Y_{N,H}(1)|>\frac{\sigma_m}{2} \Bigr) \ge \frac{9\sigma_m^4} {16\E|\tT^{\ge d}Y_{N,H}(1)|^4} \gtrsim 1.
\end{equation*}
Combining this with \eqref{e:endpoint_total_approximation} and using $\sigma_m\to\infty$, we conclude that $(\tT^{\ge d}Y_n(1))_{n\ge1}$ is not tight.
\end{proof}

    \bibliography{Refs}
    \bibliographystyle{Martin}
\end{document}